\documentclass{article}

\usepackage{amsmath,amssymb,amsthm, amscd, verbatim, setspace}
\usepackage{graphicx} 
\usepackage{url, hyperref}
\usepackage[section]{placeins}
\usepackage{tabu}
\usepackage{enumerate}

\theoremstyle{plain}
\newtheorem{thm}{Theorem} 
\numberwithin{thm}{section}
\newtheorem{lem}[thm]{Lemma}
\newtheorem{cor}[thm]{Corollary} 
\newtheorem{prop}[thm]{Proposition} 
 
\newtheorem{conj}[thm]{Conjecture}

\theoremstyle{definition}
\newtheorem{defn}[thm]{Definition}

\theoremstyle{remark}
\newtheorem{rem}[thm]{Remark}

\newcommand{\codim}[1]{\text{codim}(#1)}
\DeclareMathOperator{\@span}{span}

\title{Dense binary $PG(t-1,2)$-free matroids have critical number $t-1$ or $t$}
\author{Jonathan Tidor\footnote{Department of Mathematics, MIT; \nolinkurl{jtidor@mit.edu}}}
\date{}

\begin{document}

\maketitle

\begin{abstract}
The critical threshold of a (simple binary) matroid $N$ is the infimum over all $\rho$ such that any $N$-free matroid $M$ with $|M|>\rho2^{r(M)}$ has bounded critical number. In this paper, we resolve two conjectures of Geelen and Nelson, showing that the critical threshold of the projective geometry $PG(t-1,2)$ is $1-3\cdot2^{-t}$. We do so by proving the following stronger statement: if $M$ is $PG(t-1,2)$-free with $|M|>(1-3\cdot2^{-t})2^{r(M)}$, then the critical number of $M$ is $t-1$ or $t$. Together with earlier results of Geelen and Nelson \cite{gn} and Govaerts and Storme \cite{gs}, this completes the classification of dense $PG(t-1,2)$-free matroids.
\end{abstract}

\section{Introduction}

In this paper, the term \textit{matroid} refers to a simple binary matroid. We represent such a matroid $M$ as a set of \textit{vectors}, $E(M)$, that is a full-rank subset of  $\mathbb{F}_2^{r(M)}$ such that $0\not\in E(M)$ where $r(M)$ is the \textit{rank} of the matroid. The \textit{cardinality} of a matroid $M$, denoted $|M|$, is simply the cardinality of $E(M)$ and the \textit{critical number}, $\chi(M)$, is the smallest $k$ such that there exists a codimension-$k$ linear subspace of $\mathbb{F}_2^{r(M)}$ that is disjoint from $E(M)$. A matroid $M$ \textit{contains} another matroid $N$ if there is a linear injection $\iota\colon\mathbb{F}_2^{r(N)}\hookrightarrow\mathbb{F}_2^{r(M)}$ such that $\iota(E(N))\subset E(M)$. The \textit{projective geometry} $PG(t-1,2)$ is regarded as a matroid $M$ with $r(M)=t$ and $E(M)=\mathbb{F}_2^t\setminus\{0\}$. (Note that despite the name `projective geometry', in this paper we solely consider matroids as subsets of the vector space $\mathbb F_2^{r(M)}$.)

Our goal is to understand the relationship between density and critical number for matroids avoiding a fixed matroid $N$. In particular, we are interested in the \textit{critical threshold} of $N$: the infimum over all $\rho$ such that any $N$-free matroid $M$ with $|M|>\rho2^{r(M)}$ has bounded critical number. In this paper we determine the critical threshold of the projective geometries $PG(t-1,2)$ and characterize the critical number of $PG(t-1,2)$-free matroids with density above that threshold. This question was posed by Geelen and Nelson \cite{gn} as a generalization of the following problem in graph theory.

An example of Hajnal shows that there exist triangle-free graphs $G$ with arbitrarily large chromatic number and minimum degree arbitrarily close to $\frac13|V(G)|$ (see \cite{es}). Based on this construction, Erd\H{o}s and Simonovits \cite{es} asked whether triangle-free graphs $G$ with $\delta(G)>\frac13|V(G)|$ have bounded chromatic number. This question was solved by Thomassen \cite{t}, showing that the \textit{chromatic threshold} of the triangle is $\frac13$.

In fact, much more is known. It is easy to see that no triangle-free graphs $G$ exist with $\delta(G)>\frac12|V(G)|$. Further bounds have been derived \cite{a,j,bt} and are listed below. All of these bounds are tight.

\begin{center}
\begin{tabu}{c|c|c|c|c}
$\delta(G)/|V(G)|$ & $>\frac12$ & $>\frac25$ & $>\frac{10}{29}$ & $>\frac13$ \\\hline
$\chi(G)$ & no graphs & $\leq 2$ & $\leq 3$ & $\leq 4$
\end{tabu}
\end{center}

Goddard and Lyle \cite{gl} showed that the chromatic threshold of the complete graph $K_r$ is $\frac{2r-5}{2r-3}$. Furthermore, their result allowed them to generalize the above table to $K_r$-free graphs. Finally, Allen et al. \cite{abgkm} compute the chromatic threshold of an arbitrary graph in terms of its chromatic number. Namely, for $c\geq 3$, the chromatic threshold of a graph $H$ with $\chi(H)=c$ is one of $\left\{\frac{c-3}{c-2},\frac{2c-5}{2c-3},\frac{c-2}{c-1}\right\}$.

The analogous question for triangle-free matroids was first investigated by Davydov and Tombak in the context of linear binary codes \cite{dt}. They essentially showed that the critical threshold of the triangle is at most $\frac14$. Geelen and Nelson \cite{gn} proved a lower bound on the critical threshold of all matroids, which resolves the question for the triangle.

The following is known about the relationship between density and critical number for $PG(t-1,2)$-free matroids.

\begin{thm}[\cite{gn}]
For $t\geq 2$ and $\epsilon>0$, there exist $PG(t-1,2)$-free matroids $M$ with $|M|>(1-3\cdot2^{-t}-\epsilon)2^{r(M)}$ and arbitrarily large $\chi(M)$.
\end{thm}

\begin{thm}[Theorem 3.16 in \cite{bw}, based on \cite{dt}]
\label{thm:dt}
Any $PG(1,2)$-free matroid $M$ with $|M|>\frac142^{r(M)}$ satisfies $\chi(M)\leq 2$.
\end{thm}

\begin{thm}[\cite{gs}]
\label{thm:gs}
For $t\geq 2$, any $PG(t-1,2)$-free matroid $M$ with $|M|>(1-\frac{11}4\cdot2^{-t})2^{r(M)}$ satisfies $\chi(M)=t-1$.
\end{thm}

We use similar ideas to Govaerts and Storme's proof of Theorem \ref{thm:gs} to show that $1-3\cdot2^{-t}$ is the critical threshold of $PG(t-1,2)$ and to extend Theorem \ref{thm:dt} to higher values of $t$. The following is the main result of this paper.

\begin{thm}
\label{thm:main}
For $t\geq 2$, any $PG(t-1,2)$-free matroid $M$ with $|M|>(1-3\cdot2^{-t})2^{r(M)}$ satisfies $\chi(M)\in\{t-1,t\}$.
\end{thm}

\begin{rem}
\label{thm:rmk}
Note that any matroid $M$ with $\chi(M)\leq t-2$ satisfies $|M|\leq(1-2^{-{t-2}})2^{r(M)}$. Thus $\chi(M)\geq t-1$ for any matroid $M$ satisfying the hypotheses of Theorems \ref{thm:gs} or \ref{thm:main}.
\end{rem}

The critical threshold of other matroids is still unknown, though Geelen and Nelson have shown the following upper bound.

\begin{thm}[\cite{gn2}]
Any matroid $N$ has critical threshold at most $1-2\cdot2^{-\chi(N)}$.
\end{thm}

They also conjecture the following, of which the lower bound is known \cite{gn}.

\begin{conj}
A matroid $N$ with $\chi(N)\geq 2$ has critical threshold $1-i\cdot2^{-\chi(N)}$ for some $i\in\{2,3,4\}$.
\end{conj}

Our approach to prove Theorem \ref{thm:main} is to induct on $t$, using Theorem \ref{thm:dt} as the starting point. Our main argument in Section \ref{sec:6} shows that the $t-1$ case of the main theorem implies the $t$ case for $t\geq 6$. In Section \ref{sec:5}, we use some modification of the same methods to show that the $t=4$ case implies the $t=5$ case. Finally, in Section \ref{sec:34}, we use an alternative argument to show that Theorem \ref{thm:dt} implies the $t=3$ and $t=4$ cases.

\section{Proof of main theorem for $t\geq6$}
\label{sec:6}

Assuming the main theorem is true for some $t-1$, we suppose there exists a matroid $M$ that is a counterexample to the main theorem at this value of $t$. In Subsections \ref{ssec:hi} and \ref{ssec:2} we consider two auxiliary matroids and apply the $t-1$ case of the main theorem to them in order to derive several properties that any counterexample $M$ must satisfy. Finally, in Subsection \ref{ssec:main} we derive a contradiction from these properties. 

\subsection{Hyperplane intersection}
\label{ssec:hi}

We use the term hyperplane to refer to a linear subspace of codimension 1.

\begin{prop}
\label{thm:hyper}
Fix $t\geq3$. Assuming Theorem \ref{thm:main} for $t-1$, say $M$ is a $PG(t-1,2)$-free matroid with $|M|>(1-3\cdot2^{-t})2^{r(M)}$ and $\chi(M)>t$. Set $Y=\mathbb{F}_2^{r(M)}\setminus E(M)$. Then for any hyperplane $H<\mathbb{F}_2^{r(M)}$, we have\[|Y\setminus H|\geq 2^{r(M)-t}.\]
\end{prop}

\begin{proof}
We see that \[|E(M)\cap H|\geq |M|-|\mathbb F_2^{r(M)}\setminus H|>(1-3\cdot2^{1-t})2^{r(M)-1}.\] Therefore, by assumption, either $E(M)\cap H$ contains $PG(t-2,2)$ or $\chi(E(M)\cap H)\leq t-1$. The latter is not possible, since a subspace of $H$ of codimension at most $t-1$ that is disjoint from $E(M)\cap H$ is a subspace of the whole space of codimension at most $t$ that is still disjoint from $E(M)$.

Therefore there is some copy of $PG(t-2,2)$ in $E(M)\cap H$. Let $G$ be the dimension $t-1$ subspace of $H$ that contains this projective geometry. We use $G+v$ to denote the coset $\{g+v : g\in G\}$. Suppose there is some $v$ such that $G+v\subset E(M)$. Then there is a copy of $PG(t-1,2)$ in $E(M)$: just take our original copy of $PG(t-2,2)$ and all of $G+v$. This is impossible since $M$ is $PG(t-1,2)$-free, so every coset of $G$ must intersect $Y$. There are $2^{r(M)-t+1}$ cosets of $G$ in $\mathbb F_2^{r(M)}$ half of which are disjoint from $H$. By the above argument, each of these $2^{r(M)-t}$ cosets of $G$ off of $H$ must contain an element of $Y$, so the desired inequality follows.
\end{proof}

We can use the above proposition to give an upper bound on $|Y\cap H|$ for $H$ of any fixed codimension. We will use the following three bounds later in the argument.

\begin{cor}
\label{thm:hyperbounds}
Fix $t\geq3$. Assuming Theorem \ref{thm:main} for $t-1$, say $M$ is a $PG(t-1,2)$-free matroid with $|M|>(1-3\cdot2^{-t})2^{r(M)}$ and $\chi(M)>t$. Set $Y=\mathbb{F}_2^{r(M)}\setminus E(M)$. For any subspace $H<\mathbb{F}_2^{r(M)}$, we have
\begin{enumerate}[(i)]
\item if $\codim{H}=1$, then $|Y\cap H|<2\cdot 2^{r(M)-t}$,
\item if $\codim{H}=2$, then $|Y\cap H|<\frac32\cdot 2^{r(M)-t}$,
\item if $\codim{H}=3$, then $|Y\cap H|<\frac54\cdot 2^{r(M)-t}$.
\end{enumerate}
\end{cor}

\begin{proof}
The first inequality follows directly from Proposition \ref{thm:hyper} and the other two follow by the pigeonhole principle. Namely, for $H$ a codimension-2 subspace the pigeonhole principle implies that we can find a coset $H+v$ with $v\not\in H$ such that $|Y\cap (H+v)|\geq\frac13|Y\setminus H|$. For the sake of contradiction, suppose that $|Y\setminus H|<\frac32\cdot2^{r(M)-t}$. Then, setting $H'=H\cup (H+v)$, we have\[|Y\setminus H'|\leq\frac23|Y\setminus H|<2^{r(M)-t},\]which contradicts Proposition \ref{thm:hyper}. This proves $(ii)$.

Similarly, for the sake of contradiction, let $H$ be a codimension-3 subspace with $|Y\setminus H|<\frac74\cdot2^{r(M)-t}$. Then we can find a coset $H+v$ with $v\not\in H$ such that $|Y\cap (H+v)|\geq\frac17|Y\setminus H|$. Then setting $H'=H\cup(H+v)$, we have\[|Y\setminus H'|\leq\frac67|Y\setminus H|<\frac32\cdot2^{r(M)-t},\]which contradicts $(ii)$. This proves $(iii)$.
\end{proof}

\subsection{Doubling construction}
\label{ssec:2}

\begin{defn}
\label{defn:quot}
Let $M$ be a matroid. For $v\in\mathbb{F}_2^{r(M)}$, define the set $E(M_v)$ by $x\in E(M_v)$ if and only if both $x$ and $x+v\in E(M)$. We view $M_v$ as a matroid where $E(M_v)\subset\@span E(M_v)\cong\mathbb F_2^{r(M_v)}$ for some $r(M_v)\leq r(M)$. Note that since $M$ is simple, so is $M_v$ and $v\not\in E(M_v)$. 
\end{defn}

It is easy to compute the size of $M_v$ in terms of our original matroid.

\begin{lem}
\label{thm:pie}
Let $M$ be a matroid. Set $Y=\mathbb{F}_2^{r(M)}\setminus E(M)$. For any $v\in\mathbb{F}_2^{r(M)}$, we have\[|M_v|=2^{r(M)}-2|Y|+|Y\cap (Y+v)|,\]where $Y+v$ denotes the set $\{y+v : y\in Y\}$. In particular, if $|M|>(1-3\cdot2^{-t})2^{r(M)}$, then $|M_v|>(1-3\cdot2^{1-t})2^{r(M)}$.
\end{lem}

\begin{proof}
We have $\mathbb{F}_2^{r(M)}\setminus E(M_v)=Y\cup(Y+v)$. Thus $|\mathbb{F}_2^{r(M)}\setminus E(M_v)|=|Y|+|Y+v|-|Y\cap(Y+v)|=2|Y|-|Y\cap(Y+v)|$, as desired.
\end{proof}

The matroid $M_v$ satisfies the following two useful properties.

\begin{prop}
\label{thm:quotp}
Fix $t\geq3$. Let $M$ be a matroid and pick $v\in E(M)$. If $M_v$ contains $PG(t-2,2)$, then $M$ contains $PG(t-1,2)$.
\end{prop}

\begin{proof}
Say that $X\subset E(M_v)$ is isomorphic to $PG(t-2,2)$. By the definition of $E(M_v)$, we know that $X\cup(X+v)\cup\{v\}\subset E(M)$. Furthermore, this set is isomorphic to $PG(t-1,2)$ as long as $X$ and $X+v$ are disjoint. 

Now for any $x\in X$, we claim that $x+v\not\in X$. This is because for any $x_1,x_2\in PG(t-2,2)$, then $x_1+x_2\in PG(t-2,2)$. However, since $v\not\in E(M_v)$, we conclude that $x+v\not\in X$, completing the proof.
\end{proof}

\begin{prop}
\label{thm:quotx}
Let $M$ be a matroid and pick $v\in\mathbb{F}_2^{r(M)}$. If $\chi(M_v)=k$, then there exists a codimension-$k$ subspace $H<\mathbb{F}_2^{r(M)}$ with $v\in H$ and $|E(M)\cap H|\leq2^{r(M)-k-1}$.
\end{prop}

\begin{proof}
Let $H$ be a maximal subspace disjoint from $E(M_v)$. By assumption $\codim{H}=k$. Now if $v\not\in H$, then $H$ and $H+v$ are disjoint so $H\cup(H+v)$ is a codimension-$(k-1)$ subspace that is also disjoint from $E(M_v)$. This contradicts the maximality of $H$, so $v\in H$.

Now we claim that $|E(M)\cap H|\leq\frac12|H|$. We know that for any $x\in E(M)\cap H$, it is not the case that $x+v\in E(M)\cap H$ since then $x,x+v\in E(M_v)$. Therefore at most half of the elements of $H$ can be in $E(M)$.
\end{proof}

Combining the properties we derived above with the $t-1$ case of the main theorem and our bounds from Subsection \ref{ssec:hi}, we obtain the following useful properties of any counterexample to our main theorem.

\begin{prop}
\label{thm:quot}
Fix $t\geq3$. Assuming Theorem \ref{thm:main} for $t-1$, say $M$ is a $PG(t-1,2)$-free matroid with $|M|>(1-3\cdot2^{-t})2^{r(M)}$ and $\chi(M)>t$. Set $Y=\mathbb{F}_2^{r(M)}\setminus E(M)$. We have the following:
\begin{enumerate}[(i)]
\item There exist distinct codimension-$(t-1)$ subspaces $H_1$, $H_2$,\ldots,$H_s<\mathbb{F}_2^{r(M)}$ such that $|Y\cap H_i|\geq 2^{r(M)-t}$ for $1\leq i\leq s$ and $E(M)\subset \bigcup_{i=1}^s H_i$. Furthermore, $s>2^t-3$.
\item For all $v\in E(M)$, we have $|Y\cap (Y+v)|<2^{r(M)-t-1}$.
\end{enumerate}
\end{prop}

\begin{proof}
By Lemma \ref{thm:pie}, for $v\in E(M)$ we have that $|M_v|>(1-3\cdot2^{1-t})2^{r(M)}$. Then by assumption, either $E(M_v)$ contains $PG(t-2,2)$ or $\chi(M_v)\in\{t-2,t-1\}$. By Proposition \ref{thm:quotp}, the former cannot happen. Therefore, by Proposition \ref{thm:quotx}, there is a subspace $H$ with $v\in H$ and $\codim{H}=\chi(M_v)\in\{t-2,t-1\}$ and $|Y\cap H|\geq2^{r(M)-\chi(M_v)-1}$.

We claim that $\chi(M_v)=t-1$. If not, then $H$ has codimension $t-2$ and $|Y\cap H|\geq 2^{r(M)-t+1}$. Since $t\geq 3$, the subspace $H$ is contained in some hyperplane $H'$ that satisfies $|Y\cap H'|\geq|Y\cap H|\geq2\cdot2^{r(M)-t}$, which contradicts Corollary \ref{thm:hyperbounds} $(i)$.

This implies the first part of $(i)$. For each $v\in E(M)$ we find a corresponding $H$ as above. Let $H_1,H_2,\ldots,H_s$ be the set of distinct $H$'s. Each $H_i$ has the desired properties and since $v\in H$, we see that $E(M)\subset\bigcup_{i=1}^sH_i$.

To bound $s$, note that since $E(M)\subset\bigcup_{i=1}^sH_i$, we have\[\sum_{i=1}^s|E(M)\cap H_i|\geq|M|>(1-3\cdot2^{-t})2^{r(M)}.\]Now since $|E(M)\cap H_i|\leq2^{r(M)-t}$ for all $i$, we conclude that $s>2^t-3$.

We now can deduce $(ii)$ from the above and Theorem \ref{thm:gs}. We know that $\chi(M_v)=t-1$ and $M_v$ is $PG(t-2,2)$-free. Therefore $M_v$ cannot satisfy the hypotheses of Theorem \ref{thm:gs} since this would imply that $\chi(M_v)=t-2$. Thus we have $|M_v|\leq(1-\frac{11}4\cdot2^{1-t})2^{r(M)}$. Applying Lemma \ref{thm:pie}, we conclude that $|Y\cap(Y+v)|<2^{r(M)-t-1}$.
\end{proof}

\subsection{Proof of main theorem}
\label{ssec:main}

Proposition \ref{thm:quot} $(i)$ is almost strong enough to prove the main theorem: it says that any counterexample $M$ must have many small subspaces that each contain more than a third of the elements of $\mathbb{F}_2^{r(M)}\setminus E(M)$.

To finish the proof, we first show that there must be a pair of these subspaces with large intersection.

\begin{prop}
\label{thm:largeint}
Fix $t\geq3$. Assuming Theorem \ref{thm:main} for $t-1$, say $M$ is a $PG(t-1,2)$-free matroid with $|M|>(1-3\cdot2^{-t})2^{r(M)}$ and $\chi(M)>t$. Set $Y=\mathbb{F}_2^{r(M)}\setminus E(M)$. Then there exist distinct codimension-$(t-1)$ subspaces $H_1,H_2<\mathbb{F}_2^{r(M)}$ such that $|Y\cap H_1|,|Y\cap H_2|\geq 2^{r(M)-t}$ and $\codim{H_1\cap H_2}\leq t+2$.
\end{prop}

\begin{proof}
Applying Proposition \ref{thm:quot} $(i)$, we choose codimension-$(t-1)$ subspaces $H_1,H_2,\ldots,H_s$ with $|Y\cap H_i|\geq2^{r(M)-t}$ for each $1\leq i\leq s$. Furthermore, we have that $s\geq 4$ for $t\geq 3$.

Now we count\[|Y|\geq\sum_{1\leq i\leq4}|Y\cap H_i|-\sum_{1\leq i<j\leq 4}|Y\cap H_i\cap H_j|.\]If we assume for the sake of contradiction that $\codim{H_i\cap H_j}>t+2$ for each $1\leq i<j\leq4$, we have $|Y\cap H_i\cap H_j|\leq |H_i\cap H_j|\leq 2^{r(M)-t-3}$. This implies\[3\cdot2^{r(M)-t}>|Y|\geq4\cdot2^{r(M)-t}-6\cdot2^{r(M)-t-3}=\frac{13}42^{r(M)-t},\]a contradiction.
\end{proof}

To complete the proof, we use Proposition \ref{thm:quot} $(ii)$ to show that the intersection of these subspaces cannot contain too many elements of $\mathbb{F}_2^{r(M)}\setminus E(M)$. Finally, we show that this information contradicts the bounds that were derived in Subsection \ref{ssec:hi}.

\begin{prop}
\label{thm:tbound}
Fix $t\geq3$. Assuming Theorem \ref{thm:main} for $t-1$, say $M$ is a $PG(t-1,2)$-free matroid with $|M|>(1-3\cdot2^{-t})2^{r(M)}$ and $\chi(M)>t$. Set $Y=\mathbb{F}_2^{r(M)}\setminus E(M)$. For any codimension-$t$ subspace $H<\mathbb{F}_2^{r(M)}$, we have\[|Y\cap H|<\frac342^{r(M)-t}.\]
\end{prop}

\begin{proof}
Since $\chi(M)>t$, we know that there must be some $v\in E(M)\cap H$. Now note that $(Y\cap H)+v \subset H$. Therefore, we have
\begin{align*}
|H|&\geq|(Y\cap H)\cup((Y\cap H)+v)|\\
&=2|Y\cap H|-|Y\cap(Y+v)\cap H|\\
&\geq2|Y\cap H|-|Y\cap(Y+v)|.
\end{align*}
Now $|H|=2^{r(M)-t}$ and $|Y\cap(Y+v)|<2^{r(M)-t-1}$ by Proposition \ref{thm:quot} $(ii)$, so the desired inequality follows.
\end{proof}

\begin{thm}
\label{thm:main6}
Fix $t\geq6$. Assuming Theorem \ref{thm:main} for $t-1$, then there exist no $PG(t-1,2)$-free matroids $M$ with $|M|>(1-3\cdot2^{-t})2^{r(M)}$ and $\chi(M)>t$.
\end{thm}

\begin{proof}
Suppose for the sake of contradiction that $M$ is a counterexample. Set $Y=\mathbb{F}_2^{r(M)}\setminus E(M)$. By Proposition \ref{thm:largeint}, we can find distinct codimension-$(t-1)$ subspaces $H_1,H_2<\mathbb{F}_2^{r(M)}$ such that $t\leq\codim{H_1\cap H_2}\leq t+2$ and $|Y\cap H_1|,|Y\cap H_2|\geq2^{r(M)-t}$. We have three cases.

\textbf{Case 1:} $\codim{H_1\cap H_2}=t$

Let $H_1+H_2$ refer to the smallest subspace of $\mathbb{F}_2^{r(M)}$ containing both $H_1$ and $H_2$. We know that $\codim{H_1}+\codim{H_2}=\codim{H_1\cap H_2}+\codim{H_1+H_2}$ which implies that $\codim{H_1+H_2}=t-2$. Then
\begin{align*}
|Y\cap(H_1+H_2)|&\geq|Y\cap(H_1\cup H_2)|\\
&=|Y\cap H_1|+|Y\cap H_2|-|Y\cap(H_1\cap H_2)|\\
&>\frac54 2^{r(M)-t},
\end{align*}
where the last line follows from Proposition \ref{thm:tbound}.

If $t\geq 5$, this is a contradiction since $H_1+H_2$ has codimension $t-2\geq 3$. Therefore $H_1+H_2$ is contained in some codimension-3 subspace of $\mathbb{F}_2^{r(M)}$. However, Corollary \ref{thm:hyperbounds} $(iii)$ says that any codimension-3 subspace of $\mathbb{F}_2^{r(M)}$ contains strictly fewer than $\frac542^{r(M)-t}$ elements of $Y$.

\textbf{Case 2:} $\codim{H_1\cap H_2}=t+1$

In this case we have $\codim{H_1+H_2}=t-3$ and 
\begin{align*}
|Y\cap(H_1+H_2)|&\geq |Y\cap H_1|+|Y\cap H_2|-|Y\cap(H_1\cap H_2)|\\
&\geq|Y\cap H_1|+|Y\cap H_2|-|H_1\cap H_2|\\
&\geq\frac322^{r(M)-t}.
\end{align*}

Again for $t\geq 5$ we have that $H_1+H_2$ is contained in a codimension-2 subspace of $\mathbb{F}_2^{r(M)}$, which contradicts Corollary \ref{thm:hyperbounds} $(ii)$.

\textbf{Case 3:} $\codim{H_1\cap H_2}=t+2$

In this case, $\codim{H_1+H_2}=t-4$ and $|Y\cap(H_1+H_2)|\geq\frac742^{r(M)-t}$. This contradicts Corollary \ref{thm:hyperbounds} $(ii)$ for $t\geq 6$ since then $H_1+H_2$ is contained in a codimension-2 subspace.
\end{proof}

As noted in Remark \ref{thm:rmk}, showing that such a matroid $M$ has $\chi(M)\leq t$ is sufficient to show that $\chi(M)\in\{t-1,t\}$. Thus we have proved the desired result assuming the $t=3,4,5$ cases.

\section{Proof of main theorem for $t=5$}
\label{sec:5}

In this section we extend the results of the previous section to the $t=5$ case. Theorem \ref{thm:main6} does not apply here since case 3 of that proof requires that $t\geq 6$. However, we can still use that proof to gives us some information in the $t=5$ case.

\begin{cor}
\label{thm:int7}
Assuming Theorem \ref{thm:main} for $t=4$, say $M$ is a $PG(4,2)$-free matroid with $|M|>(1-3\cdot2^{-5})2^{r(M)}$ and $\chi(M)>5$. Set $Y=\mathbb{F}_2^{r(M)}\setminus E(M)$. If $H_1, H_2<\mathbb{F}_2^{r(M)}$ are distinct codimension-4 subspaces satisfying $|Y\cap H_1|,|Y\cap H_2|\geq 2^{r(M)-5}$, then $\codim{H_1\cap H_2}\in\{7,8\}$.
\end{cor}

\begin{proof}
For $H_1,H_2$ of codimension 4 and distinct, we have that $5\leq\codim{H_1\cap H_2}\leq 8$. As noted above, the proof of Theorem \ref{thm:main6}  rules out cases 1 and 2 even for $t=5$. The $t\geq 6$ constraint is only used in case 3. Thus we conclude that $\codim{H_1\cap H_2}\neq 5,6$, giving the desired result.
\end{proof}

We now prove a strengthened version of Proposition \ref{thm:largeint}.

\begin{prop}
\label{thm:largeints}
Assuming Theorem \ref{thm:main} for $t=4$, say $M$ is a $PG(4,2)$-free matroid with $|M|>(1-3\cdot2^{-5})2^{r(M)}$ and $\chi(M)>5$. Set $Y=\mathbb{F}_2^{r(M)}\setminus E(M)$. If $H_1,H_2,H_3,H_4<\mathbb{F}_2^{r(M)}$ are pairwise distinct codimension-4 subspaces satisfying $|Y\cap H_i|\geq 2^{r(M)-5}$ for $1\leq i\leq 4$, then there exist three pairs $1\leq i<j\leq 4$ with $\codim{H_i\cap H_j}=7$.
\end{prop}

\begin{proof}
This follows from the same argument used to prove Proposition \ref{thm:largeint}. By Corollary \ref{thm:int7}, all of the intersections have codimension 7 or 8. Suppose for the sake of contradiction that at most two of the intersections have codimension 7. Then we have
\begin{align*}
3\cdot2^{r(M)-5}>|Y|&\geq\sum_{1\leq i\leq4}|Y\cap H_i|-\sum_{1\leq i<j\leq 4}|Y\cap H_i\cap H_j|\\
&\geq4\cdot2^{r(M)-5}-4\cdot2^{r(M)-8}-2\cdot2^{r(M)-7}\\
&=3\cdot2^{r(M)-5},
\end{align*}
which is a contradiction.
\end{proof}

To complete the proof, we need the following graph-theoretic lemma.

\begin{lem}
\label{thm:graph}
Let $G$ be a graph on 30 vertices such that any four vertices of $G$ have at least three edges among them. Then there exists a vertex of $G$ with degree at least 16.
\end{lem}

\begin{proof}
Say vertex $v$ has degree $\delta$. If $\delta\geq 16$ we are done. Otherwise, note that $G$ contains a copy of the complete graph $K_{29-\delta}$. This is because for any three vertices $u_1,u_2,u_3$ that are not neighbors of $v$, all of $u_1,u_2,u_3$ must be pairwise connected.

In particular, $G$ contains a $K_{14}$. Pick two vertices $v_1,v_2$ in the $K_{14}$. Label the vertices not in our $K_{14}$ by $x_1,x_2,\ldots,x_8$ and $y_1,y_2,\ldots,y_8$. Considering the set of vertices $x_i,y_i,v_1,v_2$ for $1\leq i\leq 8$, there must be at least one edge between $\{x_i,y_i\}$ and $\{v_1,v_2\}$. Thus there are at least 8 edges from $\{v_1,v_2\}$ to outside of the $K_{14}$. Then one of $v_1,v_2$ is connected to at least 4 vertices outside of the $K_{14}$ and thus has degree at least 17.
\end{proof}

\begin{thm}
Assuming Theorem \ref{thm:main} for $t=4$, then there exist no $PG(4,2)$-free matroids $M$ with $|M|>(1-3\cdot2^{-5})2^{r(M)}$ and $\chi(M)>5$.
\end{thm}

\begin{proof}
Suppose for the sake of contradiction that $M$ is a counterexample. Let $Y$ refer to $\mathbb{F}_2^{r(M)}\setminus E(M)$. By Proposition \ref{thm:quot} $(i)$, we can find codimension-4 subspaces $H_1,H_2,\ldots,H_{30}$ such that $|Y\cap H_i|\geq2^{r(M)-5}$ for $1\leq i\leq 30$. Now Proposition \ref{thm:largeints} and Lemma \ref{thm:graph} imply that we can order the $H_i$'s such that $\codim{H_i\cap H_{17}}=7$ for $1\leq i\leq 16$. Note that since $\codim{H_i}=\codim{H_{17}}=4$ and $\codim{H_i\cap H_{17}}=7$, we see that $\codim{H_i+H_{17}}=1$, that is, $H_i+H_{17}$ is a hyperplane for $1\leq i\leq 16$.

We claim that we can find some $1\leq i<j\leq 16$ such that $H_i+H_j+H_{17}$ is a hyperplane. Consider $H_1^{\perp},H_2^{\perp},\ldots,H_{17}^{\perp}$. These are dimension 4 subspaces of $\mathbb{F}_2^{r(M)}$. We know that $H_{17}^{\perp}$ intersects each $H_i^{\perp}$ non-trivially, say at $\{0,v_i\}$.

Since $v_i\in H_{17}^{\perp}\setminus\{0\}$ and $|H_{17}^{\perp}\setminus\{0\}|=15$, by the pigeonhole principle there exist some $1\leq i<j\leq 16$ with $v_i=v_j$. Thus $H_i^{\perp}\cap H_j^{\perp}\cap H_{17}^{\perp}\neq\{0\}$, or equivalently, $H_i+H_j+H_{17}\neq\mathbb{F}_2^{r(M)}$, as desired.

Now note that $\codim{H_i\cap H_j}=7$ since $H_i+H_j\leq H_i+H_j+H_{17}\lneq\mathbb{F}_2^{r(M)}$. Finally, we have
\begin{align*}
|Y\cap(H_i+H_j+H_{17})|&\geq|Y\cap H_i|+|Y\cap H_j|+|Y\cap H_{17}|\\
&\quad-|Y\cap H_i\cap H_j|-|Y\cap H_i\cap H_{17}|-|Y\cap H_j\cap H_{17}|\\
&\geq3\cdot2^{r(M)-5}-|H_i\cap H_j|-|H_i\cap H_{17}|-|H_j\cap H_{17}|\\
&\geq3\cdot2^{r(M)-5}-3\cdot2^{r(M)-7}\\
&=\frac94\cdot2^{r(M)-5}.
\end{align*}
As $H_i+H_j+H_{17}$ is a hyperplane, this contradicts Corollary \ref{thm:hyperbounds} $(i)$.
\end{proof}

\section{Proof of main theorem for $t=3,4$}
\label{sec:34}

To finish the last two cases of the main theorem, we turn to Fourier-analytic techniques. Together with some bounds from Section \ref{sec:6} and below, standard techniques from Fourier analysis imply the $t=3$ case of the main theorem. With some more work, we use similar ideas to prove the $t=4$ case.

We start by using the classical Bose-Burton theorem to derive a bound complementary to Proposition \ref{thm:hyper}.

\begin{thm}[Bose-Burton, \cite{bb}]
\label{thm:bb}
Fix $t\geq 2$. Let $M$ be a matroid. If $|M|>(1-2\cdot2^{-t})2^{r(M)}$, then $M$ contains $PG(t-1,2)$.
\end{thm}

\begin{cor}
\label{thm:hyperupper}
Fix $t\geq3$. Assuming Theorem \ref{thm:main} for $t-1$, say $M$ is a $PG(t-1,2)$-free matroid with $|M|>(1-3\cdot2^{-t})2^{r(M)}$ and $\chi(M)>t$. Set $Y=\mathbb{F}_2^{r(M)}\setminus E(M)$. For any hyperplane $H<\mathbb{F}_2^{r(M)}$, we have\[|Y\cap H|\geq 2^{r(M)-t}.\]
\end{cor}

\begin{proof}
If not, $|E(M)\cap H|>(1-2\cdot 2^{-t})2^{r(M)-1}$. Then by Theorem \ref{thm:bb}, there is a copy of $PG(t-1,2)$ in $E(M)\cap H$.
\end{proof}

The above corollary and Proposition \ref{thm:hyper} serve to give a lower bound on $\frac1{|M|}\sum_{v\in E(M)}|Y\cap(Y+v)|$ using Fourier-analytic techniques. The following proof is based on one found in Tao and Vu \cite{tv}.

\begin{prop}
\label{thm:fourier}
Fix $t\geq3$. Assuming Theorem \ref{thm:main} for $t-1$, say $M$ is a $PG(t-1,2)$-free matroid with $|M|>(1-3\cdot2^{-t})2^{r(M)}$ and $\chi(M)>t$. Then letting $Y=\mathbb{F}_2^{r(M)}\setminus E(M)$, we have\[\frac1{|M|}\sum_{v\in E(M)}|Y\cap (Y+v)|> \frac23\frac{|Y|^2}{2^{r(M)}}.\]
\end{prop}

\begin{proof}
Proposition \ref{thm:hyper} and Corollary \ref{thm:hyperupper} together imply that for any hyperplane $H<\mathbb{F}_2^{r(M)}$, we have $\bigl||Y\cap H|-|Y\setminus H|\bigr|<\frac13|Y|$. In the language of Fourier analysis, this says that the Fourier bias of $Y$ and of $E(M)$ are less than $\frac13\cdot\frac{|Y|}{2^{r(M)}}$. To complete the proof, we introduce some notions from Fourier analysis and present a slightly strengthened version of Lemma 4.13 of \cite{tv}.

Set $Z=\mathbb{F}_2^{r(M)}$. For $f:Z\to\mathbb{R}$, we define the Fourier transform $\hat{f}:{Z}\to\mathbb{R}$ by $\hat{f}(\xi)=\frac1{|Z|}\sum_{x\in Z}f(x)(-1)^{\xi\cdot x}$ where $\xi\cdot x$ is the standard dot product in $\mathbb{F}_2^{r(M)}$ and $(-1)^{\xi\cdot x}$ is the map $\mathbb F_2\to\mathbb R$ that is 1 when $\xi\cdot x$ is 0 and $-1$ when $\xi\cdot x$ is 1.

For a set $A\subseteq Z$ define the indicator function $1_A\colon Z\to\mathbb R$ by $1_A(x)$ is 1 when $x\in A$ and 0 otherwise. For $f,g\colon Z\to\mathbb R$, define their convolution $f\ast g\colon Z\to\mathbb{R}$ by $(f\ast g)(x)=\frac1{|Z|}\sum_{y\in Z}f(y)g(x-y)$.

Three easily-verified formulae from Fourier analysis are Parseval's identity,\[\frac1{|Z|}\sum_{x\in Z}f(x)^2=\sum_{\xi\in Z}\hat f(x)^2,\] the Fourier inversion formula,\[f(0)=\sum_{\xi\in Z}\hat f(\xi),\] and the convolution identity\[\widehat{f\ast g}(\xi)=\hat f(\xi)\hat g(\xi).\]Next, note that $\hat 1_A(0)=|A|/|Z|$. Furthermore, the left-hand side of Parseval's identity is also $|A|/|Z|$ when $f$ is the indicator function $1_A$.

Set $|Y|=\alpha2^{r(M)}$. To derive the desired result, note the following two facts. First, there is a bijection between vectors $\xi\in Z\setminus\{0\}$ and hyperplanes $H<\mathbb F_2^{r(M)}$ where $\xi$ is the non-zero vector in the orthogonal complement of $H$. The vector $\xi$ satisfies $(-1)^{\xi\cdot x}$ is 1 when $x\in H$ and $-1$ otherwise. Thus from the equation
\begin{align*}
\bigl||E(M)\cap H|-|E(M)\setminus H|\bigr|
&=\bigl|(2^{r(M)-1}-|Y\cap H|)-(2^{r(M)-1}-|Y\setminus H|)\bigr|\\
&=\bigl||Y\cap H|-|Y\setminus H|\bigr|\\
&<\frac13|Y|
\end{align*}for all hyperplanes $H$, we conclude that $|\hat 1_{E(M)}(\xi)|<\frac{Y}{3\cdot 2^{r(M)}}$ for all $\xi\in Z\setminus\{0\}$.

Second, the quantity we are interested in can be expressed as follows\[\sum_{v\in E(M)} |Y\cap(Y+v)|=|\{(y_1,y_2,v)\in Y\times Y\times E(M) : y_1+y_2+v=0\}|.\]  We can write the right-hand side of the above equation as $2^{2r(M)}\cdot(1_Y\ast1_Y\ast1_{E(M)})(0)$. Then we compute
\begin{align*}
(1_Y\ast1_Y\ast1_{E(M)})(0)&=\sum_{\xi\in{Z}}\hat{1}_Y(\xi)\hat{1}_Y(\xi)\hat{1}_{E(M)}(\xi)\\
&\geq\hat 1_Y(0)^2\hat 1_{E(M)}(0)-\left|\sum_{\xi\in{Z}\setminus\{0\}}\hat{1}_Y(\xi)^2\hat{1}_{E(M)}(\xi)\right|\\
&\geq\alpha^2(1-\alpha)-\sum_{\xi\in{Z}\setminus\{0\}}\hat{1}_Y(\xi)^2|\hat{1}_{E(M)}(\xi)|\\
&> \alpha^2(1-\alpha)-\frac13\alpha\sum_{\xi\in{Z}\setminus\{0\}}\hat{1}_Y(\xi)^2\\
&= \alpha^2(1-\alpha)-\frac13\alpha\sum_{\xi\in{Z}}\hat{1}_Y(\xi)^2+\frac13\alpha\hat{1}_Y(0)^2\\
&=\alpha^2(1-\alpha)-\frac13\alpha(\alpha-\alpha^2)\\
&=\frac23\alpha^2(1-\alpha).
\end{align*}
The first line follows from the Fourier inversion formula and the convolution identity. The second line follows by taking out the contribution from $\xi=0$. The fourth line follows since $|\hat 1_{E(M)}(\xi)|<\frac13\alpha$ for $\xi\neq 0$. The sixth line follows by Parseval's identity.

Since $|M|=(1-\alpha)2^{r(M)}$ and $|Y|=\alpha2^{r(M)}$, we conclude\[\frac1{|M|}\sum_{v\in E(M)}|Y\cap (Y+v)|>\frac{1}{(1-\alpha)2^{r(M)}}\frac23\alpha^2(1-\alpha)2^{2r(M)}=\frac23\frac{|Y|^2}{2^{r(M)}},\]as desired.
\end{proof}

Now for $t=3$, this lower bound on $\frac1{|M|}\sum_{v\in E(M)}|Y\cap(Y+v)|$ contradicts the upper bound on $|Y\cap(Y+v)|$ from Proposition \ref{thm:quot} $(ii)$.

\begin{cor}
There exist no $PG(2,2)$-free matroids $M$ with $|M|>(1-3\cdot2^{-3})2^{r(M)}$ and $\chi(M)>3$.
\end{cor}

\begin{proof}
For the sake of contradiction, suppose $M$ is a counterexample. Set $Y=\mathbb{F}_2^{r(M)}\setminus E(M)$. Proposition \ref{thm:quot} $(ii)$ says that for all $v\in E(M)$, it is the case that $|Y\cap(Y+v)|<2^{r(M)-4}$ while Theorem \ref{thm:gs} gives us that $\frac{|Y|}{2^{r(M)}}\geq\frac{11}{32}$. With Proposition \ref{thm:fourier}, these bounds give\[\frac1{16}>\frac1{|M|}\sum_{v\in E(M)}\frac{|Y\cap(Y+v)|}{2^{r(M)}}>\frac23\frac{|Y|^2}{2^{2r(M)}}\geq\frac{121}{1536},\]a contradiction.
\end{proof}

This concludes the $t=3$ case. To prove the $t=4$ case, we need to improve the upper bound on $|Y\cap(Y+v)|$ given by Proposition \ref{thm:quot} $(ii)$. To do so, we use the doubling construction twice.

\begin{prop}
\label{thm:2thing}
Say $M$ is a $PG(3,2)$-free matroid with $|M|>(1-3\cdot2^{-4})2^{r(M)}$ and $\chi(M)>4$. Set $Y=\mathbb{F}_2^{r(M)}\setminus E(M)$. For any $v_1,v_2\in E(M)$ with $v_1+v_2\in E(M)$, we have\[|Y\cap (Y+v_1)|+|Y\cap (Y+v_2)|<\frac83|Y|-\frac{11}{24}2^{r(M)}.\]
\end{prop}

\begin{proof}
Consider the matroid $(M_{v_1})_{v_2}$. Since $v_2, v_1+v_2\in E(M)$, it is also the case that $v_2\in E(M_{v_1})$. Proposition \ref{thm:quotp} then implies that $(M_{v_1})_{v_2}$ is $PG(1,2)$-free. We know that $Y\cup (Y+v_1)=\mathbb F_2^{r(M)}\setminus E(M_{v_1})$, so applying this twice we see that \[Y\cup(Y+v_1)\cup(Y+v_2)\cup(Y+v_1+v_2)=\mathbb{F}_2^{r(M)}\setminus E((M_{v_1})_{v_2}).\] Therefore, by Theorem \ref{thm:gs} we know that \[|Y\cup(Y+v_1)\cup(Y+v_2)\cup(Y+v_1+v_2)|\geq\frac{11}{16}2^{r(M)}.\] If this was not true, there would be a hyperplane $H$ disjoint from $E((M_{v_1})_{v_2})$. For the same reason as in Proposition \ref{thm:quotx}, we have that $v_1,v_2\in H$, so this hyperplane then also satisfies $|Y\cap H|\geq\frac14|H|=2^{r(M)-3}$ which contradicts Corollary \ref{thm:hyperbounds} $(i)$.

Let $c_i$ be the number of cosets of $\{0,v_1,v_2,v_1+v_2\}$ that contain $i$ elements of $Y$. We have $|Y|=c_1+2c_2+3c_3+4c_4$. Say $|Y|=\alpha2^{r(M)}$ where $\alpha <\frac3{16}$. Furthermore, $|Y\cup(Y+v_1)\cup(Y+v_2)\cup(Y+v_1+v_2)|=4c_1+4c_2+4c_3+4c_4\geq\frac{11}{16}2^{r(M)}$. Together these imply $c_2+2c_3+3c_4<\left(\alpha-\frac{11}{64}\right)2^{r(M)}$.

We claim that \[|Y\cap (Y+v_1)|+|Y\cap (Y+v_2)|\leq\frac83(c_2+2c_3+3c_4),\]which is exactly what we want to show. To prove this, we calculate how much each coset of $\{0,v_1,v_2,v_1+v_2\}$ contributes to the left- and right-hand sides of this equation.

Suppose $G$ is a coset of $\{0,v_1,v_2,v_1+v_2\}$. Set $Y_G=Y\cap G$. We have four cases.

\textbf{Case 1:} $|Y_G|\leq 1$

Such a coset contributes 0 to $c_2+2c_3+3c_4$. Furthermore, clearly $|Y_G\cap(Y_G+v_1)|=|Y_G\cap(Y_G+v_2)|=0$ in this case.

\textbf{Case 2:} $|Y_G|=2$

Such a coset contributes 1 to $c_2+2c_3+3c_4$. Furthermore, exactly one of $Y_G\cap(Y_G+v_1)$, $Y_G\cap(Y_G+v_2)$, $Y_G\cap(Y_G+v_1+v_2)$ is equal to $Y_G$ and the other two are $\emptyset$. Thus $|Y_G\cap(Y_G+v_1)|+|Y_G\cap(Y_G+v_2)|\in\{0,2\}$.

\textbf{Case 3:} $|Y_G|=3$

Such a coset contributes 2 to $c_2+2c_3+3c_4$. Furthermore, it is easy to see that $|Y_G\cap(Y_G+v_1)|=|Y_G\cap(Y_G+v_2)|=2$, so $|Y_G\cap(Y_G+v_1)|+|Y_G\cap(Y_G+v_2)|=4$.

\textbf{Case 4:} $|Y_G|=4$

Such a coset contributes 3 to $c_2+2c_3+3c_4$. In this case, $Y_G\cap(Y_G+v_1)=Y_G\cap(Y_G+v_2)=Y_G$, so $|Y_G\cap(Y_G+v_1)|+|Y_G\cap(Y_G+v_2)|=8$.

Finally, it is easy to check that in each of the cases above, the contribution to $|Y\cap(Y+v_1)|+|Y\cap(Y+v_2)|$ is at most $8/3$ the contribution to $c_2+2c_3+3c_4$.
\end{proof}

\begin{thm}
There exist no $PG(3,2)$-free matroids $M$ with $|M|>(1-3\cdot2^{-4})2^{r(M)}$ and $\chi(M)>4$.
\end{thm}

\begin{proof}
For the sake of contradiction, assume that $M$ is a counterexample. Set $Y=\mathbb{F}_2^{r(M)}\setminus E(M)$. We wish to find a contradiction between Proposition \ref{thm:fourier} and Proposition \ref{thm:2thing}. For $|Y|<\frac3{16}2^{r(M)}$, we can check that\[\frac23\frac{|Y|^2}{2^{2r(M)}}>\frac43\frac{|Y|}{2^{r(M)}}-\frac{11}{48}.\] It is sufficient to show that we can pair up the elements of $E(M)$ such that if $v_1,v_2$ is a pair, $v_1+v_2\in E(M)$. This is because then we have\[\frac43\frac{|Y|}{2^{r(M)}}-\frac{11}{48}>\frac1{|M|}\sum_{v\in E(M)}\frac{|Y\cap(Y+v)|}{2^{r(M)}}>\frac23\frac{|Y|^2}{2^{2r(M)}}.\]

In fact, if we can only pair up all but one of the elements of $E(M)$, we still find a contradiction since\[\frac1{|M|}\sum_{v\in E(M)}\frac{|Y\cap (Y+v)|}{2^{r(M)}}<\frac{|M|-1}{|M|}\left(\frac43\frac{|Y|}{2^{r(M)}}-\frac{11}{48}\right)+\frac1{|M|}\left(\frac1{32}\right),\]where the bound on the left-over element comes from Proposition \ref{thm:quot} $(ii)$. Then we can check that the inequality\[\frac23\frac{|Y|^2}{2^{2r(M)}}>\frac43\frac{|Y|}{2^{r(M)}}-\frac{11}{48}+\frac1{27}\cdot\frac1{32}\]still holds for $|Y|<\frac3{16}2^{r(M)}$. Since the inequality $4<\chi(M)\leq r(M)$ implies that $|M|>\frac{13}{16}2^{r(M)}>26$, this is sufficient.

To prove that we can pair up all or all but one elements of $E(M)$, we use the following graph-theoretic argument. Construct a graph whose vertices are elements of $E(M)$ and where two vertices are connected by an edge if their sum is also in $E(M)$. Note that a vertex $v$ has degree $|M_v|>\frac582^{r(M)}>|E(M)|/2$. By Dirac's theorem, $E(M)$ has a Hamiltonian cycle. Picking alternating edges from this cycle gives the desired matching.
\end{proof}

\section*{Acknowledgments}

This research was conducted at the University of Minnesota Duluth REU and was supported by NSF grant 1358659 and NSA grant H98230-13-1-0273. The author thanks Joe Gallian for suggesting the problem as well as Levent Alpoge, Noah Arbesfeld, and an anonymous referee for helpful comments on the manuscript.

\end{document}